\documentclass{amsart}
\usepackage{graphicx} 
\usepackage{subfig}
\usepackage{tabularx}
\usepackage[british]{babel}

\usepackage{transparent} 
\usepackage{eso-pic} 
\usepackage{subfig} 
\usepackage{tikz} 
\usetikzlibrary{}
\graphicspath{{./Images/}} 
\usepackage{caption} 
\usepackage{xcolor} 
\usepackage{amsthm,thmtools,xcolor} 
\usepackage{float}
\usepackage{todonotes}

\usepackage{geometry}

\usepackage[utf8]{inputenc}
\usepackage[T1]{fontenc}
\usepackage{mathtools}
\usepackage[thinc]{esdiff}

\usepackage{amsmath}
\usepackage{amsthm}
\usepackage{amssymb}
\usepackage{amsfonts}
\usepackage{bm}
\usepackage[overload]{empheq} 
\usepackage{fix-cm} 
\usepackage{algorithm}
\usepackage{algorithmic}

\newtheorem{lemma}{Lemma}[section]  
\newtheorem{theorem}[lemma]{Theorem}  
\newtheorem{proposition}[lemma]{Proposition}
\newtheorem{remark}[lemma]{Remark}

\newtheorem{corollary}[lemma]{Corollary}
\newtheorem{definition}[lemma]{Definition}

\usepackage[colorlinks=true,linkcolor=black,anchorcolor=black,citecolor=black,filecolor=black,menucolor=black,runcolor=black,urlcolor=black]{hyperref} 
\usepackage{cleveref}
\crefname{lemma}{Lemma}{Lemmas}
\crefname{theorem}{Theorem}{Theorems}
\crefname{proposition}{Proposition}{Propositions}
\crefname{remark}{Remark}{Remarks}
\crefname{problem}{Problem}{Problems}
\crefname{corollary}{Corollary}{Corollaries}
\crefname{definition}{Definition}{Definitions}

\crefname{section}{Section}{Sections}
\usepackage{graphicx}
\usepackage{float}
\usepackage[square, numbers, sort&compress]{natbib} 

\def\R{{\mathbb{R}}}
\def\erre{{\mathbb{R}}}
\def\enne{{\mathbb{N}}}
\def\esse{{\mathbb{S}}}

\title{Arnold stability and rigidity in Zeitlin's model of hydrodynamics}
\author{Luca Melzi}
\address{Department of Mathematics, Imperial College London}
\email{l.melzi24@imperial.ac.uk}
\author{Klas Modin}
\address{Department of Mathematical Sciences, Chalmers University of Technology and University of Gothenburg}
\email{klas.modin@chalmers.se}
\date{\today}


\begin{document}

\begin{abstract}
    Zeitlin's model is a discretisation of the 2-D Euler equations that preserves the underlying geometric structure.
    This feature makes it suitable for studying the qualitative behaviour of the dynamics.
    Here, we utilise Arnold's geometric approach to prove Lyapunov stability of steady states in Zeitlin's model.
    Furthermore, we show that such Arnold stable stationary solutions are subject to a rigidity condition that enforces a specific form of the matrix describing the state.
    Our argument relies on matrix theory and is therefore detached, and conceptually different, from the nonlinear stability analysis as developed for the 2-D Euler equations.
    Nevertheless, our results concur with those known for the 2-D Euler equations, which hints at links between matrix theory and nonlinear PDE techniques.
    Furthermore, our results show that the Zeitlin's model, as a numerical discretisation, is reliable for studying stationary solutions.
\end{abstract}

\maketitle

\section{Introduction}

The motion of an ideal, incompressible fluid evolving on a two-dimensional (2-D) Riemannian manifold $M$ is described by the Euler equations in vorticity form
\begin{equation}
      \dot\omega + \{\psi,\omega\}=0,\quad\Delta\psi=\omega,
      \label{eq:euler}
\end{equation}
equipped with initial conditions and, if needed, boundary conditions.
Here, $\omega$ and $\psi$ are the vorticity and stream function, and we denote the Poisson bracket by $\{\omega,\psi\}=\nabla^\perp \omega\cdot\nabla \psi$, where $\nabla^\perp$ is the skew gradient relative to the symplectic structure induced by the area form. 

Experimental and numerical evidence suggests a tendency of the fluid to organise (after an intermediate transient) into large-scale coherent structures, corresponding to special states such as steady states or periodic or quasi-periodic orbits.
Based thereon, 
Šverák~\cite{Sverak2012} and Shnirelman~\cite{Shnirelman2013} formalised the setting mathematically and gave conjectures which can be summarised as follows: generic orbits experience loss of compactness in $L^2$, relaxing in infinite time into special states whose orbits are compact in $L^2$ (see \cite[Section 3.4]{Drivas2023}).
Whereas these statements concern infinite-dimensional mathematical concepts, carefully conducted numerical simulations provide statistical counterparts to the statements supporting the conjectures (see Figure~\ref{fig:01}).

In 1966, Arnold~\cite{Ar1966} showed that solutions of the equations~\eqref{eq:euler}, or more generally incompressible Euler equations on a Riemannian manifold of any dimension, describe geodesics on the infinite-dimensional group of volume-preserving diffeomorphisms equipped with a right-invariant Riemannian metric.
This insight gave new tools for studying the behaviour of solutions, for example, in terms of sectional curvature, which Arnold showed is mostly negative~\cite{ArKh1998}, or for existence theory as developed by Ebin and Marsden~\cite{EbMa1970}.
Although Arnold's paper focuses on hydrodynamics and the Euler equations, the developments are carried out in an abstract framework as geodesic motion on a Lie group equipped with a left- or right-invariant Riemannian metric.
This setting gives rise to Hamiltonian Lie--Poisson systems on the dual of the Lie algebra of the Lie group.
Such systems are today at the centre stage of the field called geometric hydrodynamics (cf.~\cite{ArKh1998}).
In his 1966 paper, Arnold also derived conditions for nonlinear (Lyapunov) stability of stationary solutions.

\begin{figure}
    \captionsetup[subfigure]{labelformat=empty}
    \captionsetup[subfigure]{width=0.24\textwidth}
    \captionsetup[subfigure]{font=small}
    \centering
    \subfloat[initial]{\includegraphics[width=0.24\textwidth]{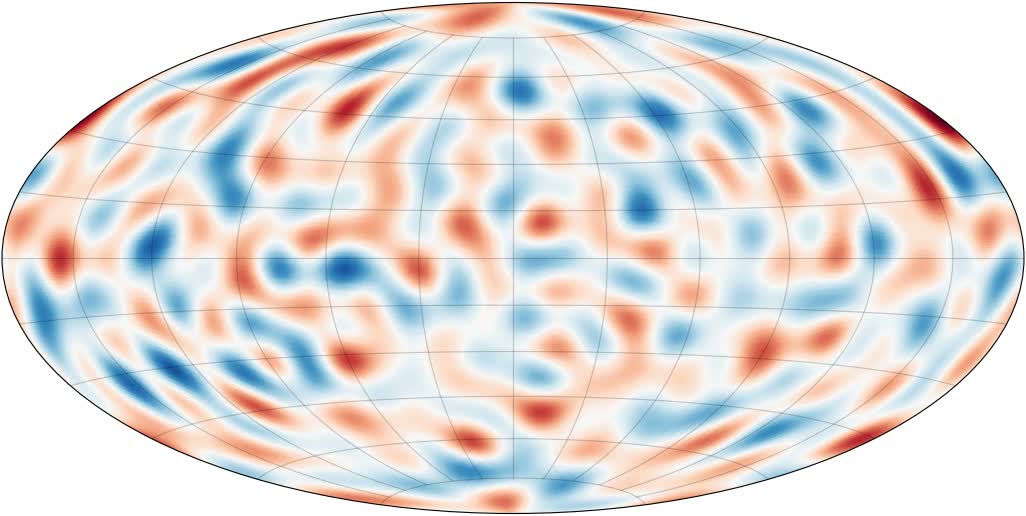}}\,
    \subfloat[first intermediate]{\includegraphics[width=0.24\textwidth]{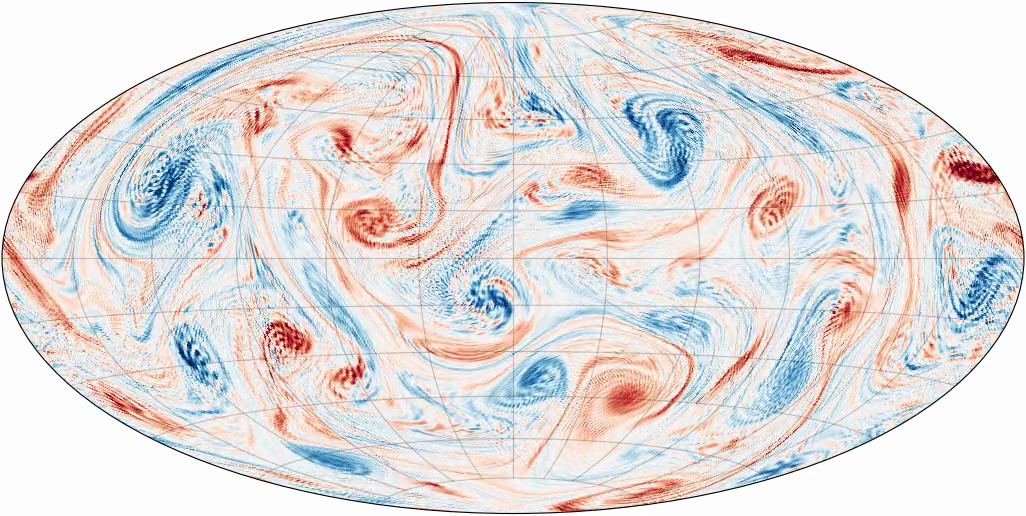}}\,
    \subfloat[second intermediate]{\includegraphics[width=0.24\textwidth]{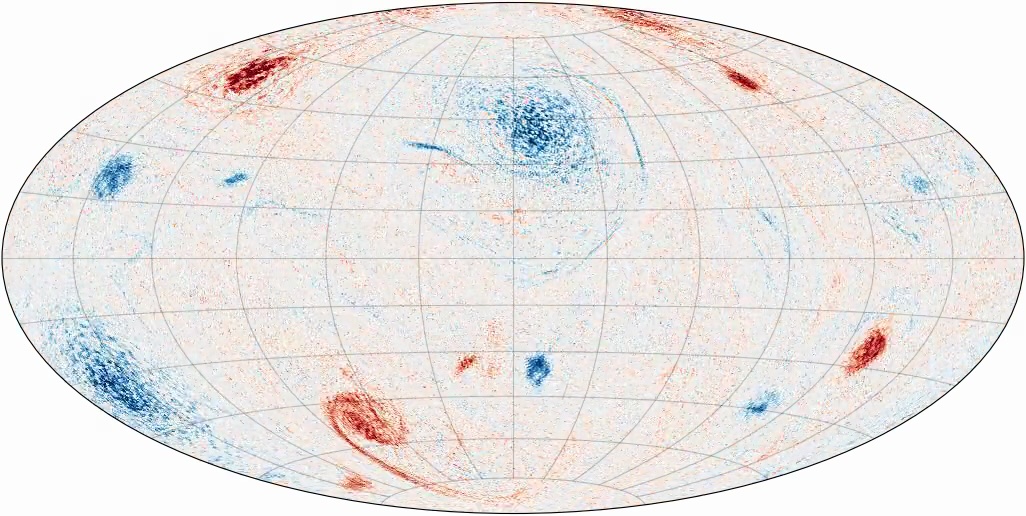}}\,
    \subfloat[long-term]{\includegraphics[width=0.24\textwidth]{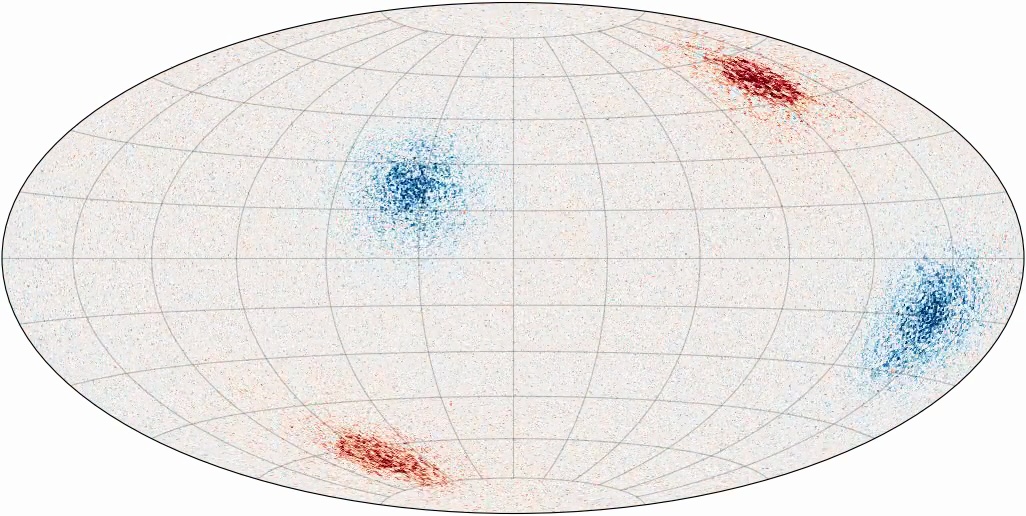}}\, 
    \caption{Snapshots of a typical time evolution of vorticity in Euler's equations on the sphere. A smooth initial vortex configuration undergoes intermediate chaotic stretch-and-fold motion, where vorticity region of equal sign (red or blue) tend to mix into larger vortex condensates.
    At some point, the mixing stops, with a few remaining and interacting coherent vortex structures.
    These numerical results are consistent with the conjectured long-term behaviour~\cite{Sverak2012,Shnirelman2013}.
    For details of the numerical simulations, see the study by Modin and Viviani~\cite{MoVi2026}.}
    \label{fig:01}
\end{figure}

In the context of Šverák's and Shnirelman's ideas about the long-term behaviour, together with Arnold's geometrical description of hydrodynamics, Zeitlin's model~\cite{Ze1991,Ze2004} provides a complementary tool to study the long-term behaviour of 2-D ideal incompressible fluids.
Based on quantisation theory by Hoppe~\cite{Hoppe1982}, it is the only known discretisation that
captures the geometrical Lie--Poisson structure of the 2-D Euler flows \cite{Hoppe1989}, which, for example, leads to discrete analogues of the conservation of the Casimirs.
More precisely, for $n\in\mathbb{N}$, $n\geq2$, Zeitlin's model on the sphere is the Lie--Poisson system on $\mathfrak{su}(n)^* \simeq \mathfrak{su}(n)$ given by
\begin{equation}
    \dot{W}+\frac{1}{\hbar}[P,W]=0,\quad\Delta_nP=W.
    \label{eq:Zeitlin} 
\end{equation}
Here, $\hbar:=2/\sqrt{n^2-1}$, the bracket $[\cdot,\cdot]$ denotes the matrix commutator, and $\Delta_n\colon \mathfrak{su}(n)\to \mathfrak{su}(n)$ is the Hoppe--Yau Laplacian~\cite{Hoppe1998} relating the vorticity matrix $W$ to the stream matrix $P$.
The Hamiltonian is given by
\begin{equation*}
    H(W)=\frac{1}{2}\langle P,W\rangle,
\end{equation*}
and is a conserved quantity, with the angular brackets $\langle\cdot,\cdot\rangle$ denoting the scaled Frobenius inner product
\begin{equation}
    \langle X,Y\rangle=\frac{1}{n}\operatorname{tr}(X^\dagger Y)\quad\forall X,Y\in\mathfrak{u}(n).
    \label{eq:Frobenius_inner_prod}
\end{equation}
The Zeitlin model retains a fundamental property of 2-D Euler flows, namely, it conserves the Casimirs
\begin{equation*}
    C_k^n(W)=\frac{1}{n}\operatorname{tr}((-\mathrm{i}W)^k),\quad k\in\mathbb{N}.
\end{equation*}
In particular, taking $k=2$ yields conservation of the (squared) Frobenius norm $\|\cdot\|^2=\langle\cdot,\cdot\rangle$.
Introducing the matrices $X_1,X_2,X_3\in\mathfrak{u}(n)$ as an irreducible representation of $\mathrm{SO}(3)$ \cite{Hoppe1998}, we can show that the angular momenta
\begin{equation}
    \boldsymbol{L}(W)=\{\langle W,X_\alpha\rangle\}_{\alpha=1}^3
    \label{eq:angmom}
\end{equation}
are additional conserved quantities due to the $\mathrm{SO}(3)$ symmetry, see \cite[Section 2.2]{delgadino2025quantitativestabilityrossbyhaurwitzwaves}. 
It is worth mentioning that the matrices $\{X_\alpha\}_{\alpha=1}^3$ span the first (non-trivial) eigenspace of $\Delta_n$.
Thus, $\mathbb{P}_1W$ is frozen by the evolution, with $\mathbb{P}_\ell$ denoting the projection onto the $\ell$-th eigenspace.
This property is the discrete analogue of the conservation of the angular momentum for the Euler equations on $\mathbb{S}^2$.
All the developments in this paper concerns Zeitlin's model on the sphere $\mathbb{S}$, namely the equations~\eqref{eq:Zeitlin}.
For the flat torus $\mathbb{T}^2$, which was the original model of Zeitlin~\cite{Ze1991}, a convergence analysis was carried out by Gallagher~\cite{Gallagher2002}. 
Finally, let us remark that the flow defined by \eqref{eq:Zeitlin} belongs to the broader class of isospectral flows, where the eigenvalues of $W$ are preserved by the evolution (see, for instance, the survey by Watkins~\cite{Wa1984}).
From the point-of-view of Lie--Poisson geometry, the isospectral property is the manifestation of preservation of coadjoint orbits.
 
The present paper deals with steady states of \eqref{eq:Zeitlin}, namely matrices $W_0,P_0\in\mathfrak{su}(n)$ fulfilling
\begin{equation}
    [P_0,W_0]=0,\quad W_0=\Delta_nP_0.
    \label{eq:Zeitlin_steady}
\end{equation}
A particular case for which \eqref{eq:Zeitlin_steady} is verified is when
\begin{equation}
    W_0=\Delta_nP_0=\mathrm{i} f(-\mathrm{i} P_0),
    \label{eq:Zeitlin_steady_f}
\end{equation}
for some real analytic function $f:\mathbb{R}\to\mathbb{R}$.

Constantin and Germain~\cite{Costantin2022} analysed the Euler equations \eqref{eq:euler} on the unit sphere $\mathbb{S}^2$.
They considered the continuous analogues of \eqref{eq:Zeitlin_steady_f}, i.e., steady states of the form $\omega_0=f(\psi_0)$ for $f\in\mathcal{C}^1(\mathbb{R})$, and they showed Lyapunov stability assuming that $f'>-6$ everywhere.
Moreover, they proved that non-trivial steady states of such kind exist only for $-6<f'<-2$, which are all zonal.

The main goal of this paper is to establish analogue results for Zeitlin's model \eqref{eq:Zeitlin}.
The motivation is twofold.
First, as an approximation, it is essential to understand which aspects of the 2-D Euler flow carry over to the Zeitlin model.
Indeed, the richer the set of analogies, the more useful the approximation.
In addition to the model's premise, namely preservation of the Lie--Poisson geometry, previously established analogies include Lie--Poisson preserving time integration schemes~\cite{MoVi2020c}, qualitatively correct spectral scaling laws~\cite{CiViLuMoGe2022}, convergence of sectional curvature and Jacobi fields~\cite{MoPe2024}, and convergence of solutions~\cite{MoVi2026}.
Here, we thus extend this list further by investigating Lyapunov stability.

But there is also a second motivation, which goes beyond Zeitlin's model as a numerical discretisation.
Namely, Zeitlin's model provides a link between the analysis of the 2-D Euler equations and matrix theory (Lie theory, random matrices, representation theory, et cetera).
Since matrix theory is such a well-developed part of mathematics, it might provide new tools for the 2-D Euler equations, complementary to the traditional PDE analysis tools.
An example is the \emph{canonical decomposition}~\cite{MoVi2022} of the vorticity field, which seems to naturally capture the large and small scales in the long-term statistical behaviour.
This decomposition came from the interpretation of the vorticity field as a complex matrix.
Another example is vortex mixing, where the matrix point-of-view leads to new ideas (cf.~\cite{MoVi2026}).
From these perspectives, we remark that the tools we use in this paper to prove Arnold stability are not based on the infinite-dimensional PDE analysis used by Constantin and Germain.
Instead, they are based on matrix theory and are fully formulated within the realm of Zeitlin's model itself, without considering its connection to the infinite-dimensional 2-D Euler flow.
Our starting point is thus the abstract (Lie theoretic) formulation of Arnold's geometric method for stability, which we review in Section~\ref{sec:Arnold_method}.



Consider a steady state $W_0,P_0\in\mathfrak{su}(n)$ satisfying equation \eqref{eq:Zeitlin_steady}.
Denote by $\{\mathrm{i}p_j\}_{j=1}^n$ the eigenvalues of $P_0$ and by $\{\mathrm{i}w_j\}_{j=1}^n$ the eigenvalues of $W_0$ (ordered with respect to a common eigenbasis), and define the quantity
\begin{equation*}
    L:=\min_{\substack{m\in\{0,...,n-1\}\\ j\in\{1,...,n-m\}}}\left\{\frac{w_{j+m}-w_j}{p_{j+m}-p_j}\right\}.
\end{equation*}
Regarding stability of the steady state, our main result is collected in the following theorem. 
\begin{theorem}[Stability] \label{thm:stability_informal}
    If $L>-6$, then $W_0,P_0$ is Lyapunov stable in the Frobenius norm.
    In particular, if $W_0,P_0$ is a steady state of type \eqref{eq:Zeitlin_steady_f} with $f'>-6$ everywhere, then it is Lyapunov stable in the Frobenius norm.
\end{theorem}
The statement of this stability theorem is collected in Theorem \ref{thm:stability} in Section \ref{sec:application_to_Zeitlin}.
It is worth noticing that the condition $L>-6$ is associated with the definiteness of the Hessian of the Hamiltonian restricted to the orbit (cf. Section \ref{sec:Arnold_method}).
We further show that such Hessian is never positive definite (unless in the case of the trivial identically null steady state), and that negative definiteness forces the matrix to be diagonal.
These rigidity results are collected in the following theorem, here stated in an informal way.
\begin{theorem}[Rigidity] \label{thm:rigidity_informal}
    If $L>-6$, then the vorticity matrix $W_0$ is diagonal up to rotations.
    In particular, if $L>-2$, then $W_0=0$.
\end{theorem}
The precise results are found in Propositions \ref{prop:rigidity_-6} and \ref{prop:rigidity_-2} in Section \ref{sec:rigidity}.
For steady states of the type in equation~\eqref{eq:Zeitlin_steady_f}, one can here also replace the condition on $L$ with a condition on $f'$.

The paper is organised as follows.
In Section \ref{sec:Arnold_method}, we outline the Arnold method to derive a general condition for Lie--Poisson dynamics.
This method is then applied to the Zeitlin model in Section \ref{sec:application_to_Zeitlin}, proving Theorem \ref{thm:stability_informal}.
Finally, Section \ref{sec:rigidity} is devoted to the proof of the rigidity results stated in Theorem \ref{thm:rigidity_informal}.

\section{Background on nonlinear stability of Lie--Poisson systems} \label{sec:Arnold_method}

In this section, we recall the abstract theory for Lyapunov stability of Lie--Poisson systems; for more details we refer to the treatise by Arnold and Khesin~\cite{ArKh1998}.
Consider a Lie group $G$ with associated Lie algebra $\mathfrak{g}$ and Lie bracket $[\cdot,\cdot]_\mathfrak{g}$.
The Lie--Poisson dynamics on the dual $\mathfrak{g}^*$ for the Hamiltonian $H\in\mathcal{C}^\infty(\mathfrak{g}^*)$ is then given by
\begin{equation}\label{eq:LP}
      \dot\omega = \operatorname{ad}^*_{dH(\omega)}\omega,
\end{equation}
where $\operatorname{ad}^*_\psi\colon \mathfrak{g}^*\to\mathfrak{g}^*$ is determined by
\begin{equation*}
      \langle \operatorname{ad}^*_\psi\omega, \xi\rangle_{\mathfrak{g}^*,\mathfrak{g}} = \langle \omega, \underbrace{\operatorname{ad}_\psi\xi}_{[\psi,\xi]_\mathfrak{g}}\rangle_{\mathfrak{g}^*,\mathfrak{g}}.
\end{equation*}
The equation \eqref{eq:LP} evolves on the coadjoint orbit $\mathcal{O}_\omega = \{\operatorname{Ad}_{g}^*\omega \mid g \in G \}$, where $\operatorname{Ad}_g^*\colon \mathfrak{g}^*\to \mathfrak{g}^*$ is defined by
\begin{equation*}
      \langle \operatorname{Ad}_g^*\omega, \xi \rangle_{\mathfrak{g}^*,\mathfrak{g}} = \langle \omega, \operatorname{Ad}_g \xi \rangle_{\mathfrak{g}^*,\mathfrak{g}}.
\end{equation*}
Here, $\operatorname{Ad}_g$ denotes the action of $G$ on its Lie algebra $\mathfrak{g}$.
Let $\omega_0\in \mathfrak{g}^*$ be a stationary point for \eqref{eq:LP}, i.e., $\operatorname{ad}^*_{dH(\omega_0)}\omega_0 = 0$. 
Arnold's method for stability~\cite{Ar1966} consists in checking that $H|_{\mathcal{O}_{\omega_0}}$, the Hamiltonian restricted to $\mathcal{O}_{\omega_0}$, is strictly convex or concave in a neighbourhood of $\omega_0$.
If so, the Hamiltonian is a Lyapunov function for the system restricted to the coadjoint orbit.
Thus, we need to compute the Hessian of $H\vert_{\mathcal{O}_{\omega_0}}$.
The strategy is to use that the coadjoint orbit $\mathcal{O}_{\omega_0}$ is (locally) parametrised by $\operatorname{Ad}^*_{\operatorname{exp}\xi}\omega_0$ for $\xi\in\mathfrak{g}$.
More precisely, following \cite{Ar1966}, we employ such parametrisation to show the following expressions for $H$ and its Hessian $Q$.
\begin{proposition} \label{prop:Arnold}
    Given a stationary point $\omega_0\in \mathfrak{g}^*$ for \eqref{eq:LP}, we have
    \begin{equation}
        H(\operatorname{Ad}^*_{\operatorname{exp}(\epsilon \xi)}\omega_0)-H(\omega_0)=\frac{\epsilon^2}{2}Q(\xi)+\mathcal{O}(\epsilon^3),
        \label{eq:2nd_order_Taylor_exp_for_H}
    \end{equation}
    where $Q$ is the following quadratic form on $\mathfrak{g}$:
    \begin{equation*}
        Q(\xi) = \langle D^2H(\omega_0)\operatorname{ad}^*_\xi\omega_0,\operatorname{ad}^*_\xi\omega_0\rangle_{\mathfrak{g}^*,\mathfrak{g}} + \langle \mathrm{ad}_\xi\big(dH(\omega_0)\big), \operatorname{ad}_\xi^*\omega_0 \rangle_{\mathfrak{g}^*,\mathfrak{g}}.
    \end{equation*}
\end{proposition}
Before the proof, observe that $Q$ depends on $\xi\in\mathfrak{g}$ only through the associated variation $\operatorname{ad}^*_{\xi}\omega_0$.
\begin{remark}
      The quadratic form $Q$ is well-defined on $T_{\omega_0}^*\mathcal{O}_{\omega_0} \simeq \mathfrak{g}/\mathrm{stab}(\omega_0)$.
      Indeed, if $\operatorname{ad}^*_{\xi_1}(\omega_0) = \operatorname{ad}^*_{\xi_2}(\omega_0)$ then
      \begin{align*}
            Q(\xi_1) =& \langle \operatorname{ad}_{\xi_1}\psi_0, \operatorname{ad}^*_{\xi_1}\omega_0\rangle_{\mathfrak{g}^*,\mathfrak{g}} = \langle [\xi_2,[\xi_1,\psi_0]_\mathfrak{g}]_\mathfrak{g},\omega_0 \rangle_{\mathfrak{g}^*,\mathfrak{g}}\\
            =& \langle [\xi_1,[\xi_2,\psi_0]_\mathfrak{g}]_\mathfrak{g} + [[\xi_2,\xi_1]_\mathfrak{g},\psi_0]_\mathfrak{g},\omega_0 \rangle_{\mathfrak{g}^*,\mathfrak{g}} = \langle \operatorname{ad}_{\xi_2}\psi_0, \operatorname{ad}^*_{\xi_2}\omega_0 \rangle_{\mathfrak{g}^*,\mathfrak{g}} = Q(\xi_2),
      \end{align*}
      where we used that $\operatorname{ad}^*_{\psi_0}\omega_0 = 0$.
      Thus, we might write $Q(\operatorname{ad}^*_{\xi}\omega_0)$ instead of $Q(\xi)$.
\end{remark}
It is also worth mentioning that, in the special case of a geodesics motion, Proposition \ref{prop:Arnold} can be re-stated as follows.
\begin{corollary}[Euler--Arnold equations] \label{cor:Arnold}
    In the setting of Proposition \ref{prop:Arnold}, let
    \begin{equation}
        H(\omega) =\frac{1}{2} \langle A^{-1}\omega,\omega\rangle_{\mathfrak{g}^*,\mathfrak{g}}
        \label{eq:geodesic_motion}
    \end{equation}
    for some positive inertia operator $A\colon \mathfrak{g}\to \mathfrak{g}^*$.
    Then, \eqref{eq:2nd_order_Taylor_exp_for_H} holds with
    \begin{equation}
        Q(\xi) = \langle A^{-1} \operatorname{ad}^*_\xi\omega_0, \operatorname{ad}^*_\xi\omega_0\rangle_{\mathfrak{g}^*,\mathfrak{g}} + \langle \operatorname{ad}_\xi A^{-1}\omega_0, \operatorname{ad}_\xi^*\omega_0 \rangle_{\mathfrak{g}^*,\mathfrak{g}}.
        \label{eq:def_of_Q_corollary}
    \end{equation}
\end{corollary}
Now that we have an explicit expression for $Q$ available, the Arnold method for stability consists in checking that such $Q$ is positive or negative definite.
This is done in Section \ref{subsec:Euler} for the Euler equations on $\mathbb{S}^2$ and then in Section \ref{sec:application_to_Zeitlin} for the Zeitlin model.
Before that, we give the proof of Proposition \ref{prop:Arnold}.
\begin{proof}[Proof of Proposition \ref{prop:Arnold}]
    Since $\operatorname{Ad}_g$ provides a representation of $G$ on $\mathfrak{g}$, we have the formula
    \begin{equation*}
        \operatorname{Ad}_{\operatorname{exp}\xi}\psi = \operatorname{exp}(\operatorname{ad}_\xi)\psi,    
    \end{equation*}
    which yields the Taylor expansion
    \begin{equation*}
        \operatorname{Ad}_{\operatorname{exp}(\epsilon\xi)}\psi = \sum_{k=0}^\infty \frac{\epsilon^k}{k!} (\operatorname{ad}_\xi)^k\psi.
    \end{equation*}
    Consequently, we also get
    \begin{equation}
        \operatorname{Ad}^*_{\operatorname{exp}(\epsilon\xi)}\omega = \sum_{k=0}^\infty \frac{\epsilon^k}{k!} (\operatorname{ad}^*_\xi)^k\omega.
        \label{eq:Taylor_exp_Ad_star}
    \end{equation}
    Differentiating the function $\epsilon \mapsto H(\operatorname{Ad}^*_{\operatorname{exp}(\epsilon \xi)}\omega_0)$ we get
    \begin{align*}
        \diff{^2}{\epsilon^2} H(\operatorname{Ad}^*_{\operatorname{exp}(\epsilon \xi)}\omega_0) =& \diff{}{\epsilon}\left( 
        H'(\operatorname{Ad}^*_{\operatorname{exp}(\epsilon \xi)}\omega_0)[\operatorname{ad}^*_\xi\omega_0 + \epsilon (\operatorname{ad}^*_\xi)^2\omega_0 + \mathcal{O}(\epsilon^2)]
        \right) \\
        =&H''(\operatorname{Ad}^*_{\operatorname{exp}(\epsilon \xi)}\omega_0)[\operatorname{ad}^*_\xi\omega_0 + \mathcal{O}(\epsilon),\operatorname{ad}^*_\xi\omega_0 + \mathcal{O}(\epsilon)] \\
        &+
        H'(\operatorname{Ad}^*_{\operatorname{exp}(\epsilon \xi)}\omega_0)[(\operatorname{ad}^*_\xi)^2\omega_0 + \mathcal{O}(\epsilon)],
    \end{align*}
    which yields the Taylor expansion
    \begin{align}
    \begin{split}
        H(\operatorname{Ad}^*_{\operatorname{exp}(\epsilon \xi)}\omega_0)=&H(\omega_0) + \epsilon H'(\omega_0)[\operatorname{ad}^*_\xi\omega_0] \\
        &+  \frac{\epsilon^2}{2}\Big( H''(\omega_0)[\operatorname{ad}^*_\xi\omega_0,\operatorname{ad}^*_\xi\omega_0] + H'(\omega_0)[(\operatorname{ad}^*_\xi)^2\omega_0]\Big)+\mathcal{O}(\epsilon^3)\\
        =&H(\omega_0) + \epsilon \langle dH(\omega_0), \operatorname{ad}^*_{\xi}\omega_0 \rangle_{\mathfrak{g}^*,\mathfrak{g}}\\
        &+ 
        \frac{\epsilon^2}{2}\Big( \langle D^2H(\omega_0)\operatorname{ad}^*_\xi\omega_0,\operatorname{ad}^*_\xi\omega_0\rangle_{\mathfrak{g}^*,\mathfrak{g}} + \langle dH(\omega_0), (\operatorname{ad}_\xi^*)^2\omega_0 \rangle_{\mathfrak{g}^*,\mathfrak{g}} \Big)+\mathcal{O}(\epsilon^3).
    \label{eq:Taylor_expansion}
    \end{split}
    \end{align}
    Since
    \begin{equation*}
        \langle \psi, \operatorname{ad}^*_\xi\omega\rangle_{\mathfrak{g}^*,\mathfrak{g}} = \langle [\xi,\psi],\omega\rangle_{\mathfrak{g}^*,\mathfrak{g}} = \langle \xi, -\operatorname{ad}^*_{\psi}\omega \rangle_{\mathfrak{g}^*,\mathfrak{g}},
    \end{equation*}
    the second term on the right-hand side of the Taylor expansion \eqref{eq:Taylor_expansion} vanishes, thus concluding the proof of Proposition \ref{prop:Arnold}.
\end{proof}

\subsection{Euler's equations on the sphere} \label{subsec:Euler}

Let $M$ be a Riemannian 2-manifold and $G = \mathrm{Diff}_\mu(M)$ be the space of all volume preserving diffeomorphisms on $M$.
If the first co-homology of $M$ is trivial, we can identify the Lie algebra $\mathfrak{g}$ with $\mathcal{C}^\infty(M)/\mathbb{R}$ equipped with the Poisson bracket $\{\cdot,\cdot\}$.
The smooth dual $\mathfrak{g}^*$ is then $\mathcal{C}^\infty_0(M)$ with dual pairing via the $L^2$ inner product.
Thus, $\operatorname{ad}_\xi\psi =-\{\xi,\psi\}$ and $\operatorname{ad}^*_\xi\omega = \{ \xi, \omega\}$.
Furthermore, the inertia operator is given by $A = -\Delta \colon \mathcal{C}^\infty(M)/\mathbb{R}\to \mathcal{C}^\infty_0(M)$.
In this setting, we retrieve the 2-D Euler equations \eqref{eq:euler} as a particular case of the Lie--Poisson dynamics \eqref{eq:LP}.
For $\xi \in \mathcal{C}^\infty(M)/\mathbb{R}$, we have that the quadratic form in \eqref{eq:def_of_Q_corollary} reads
\begin{equation}
    Q(\xi) = \langle -\Delta^{-1} \{\xi,\omega_0 \}, \{\xi,\omega_0 \} \rangle_{L^2}-\langle \{\xi, -\underbrace{\Delta^{-1}\omega_0}_{\psi_0} \}, \{\xi,\omega_0 \} \rangle_{L^2}.
    \label{eq:quadratic_form_Q_for_Euler}
\end{equation}
We are now interested in finding conditions for the definiteness of $Q$ in the case of a specific steady state $\omega_0$ whose stream function is given via functional relation.
\begin{proposition} \label{prop:stability_geq-2}
    Let $\omega_0 = \Delta\psi_0 = f(\psi_0)$ for some analytic function $f:\erre\to\erre$ be a stationary solution to the 2-D Euler equations \eqref{eq:euler} on the manifold $M$.
    Denote by $\lambda_{\text{min}}>0$ the smallest eigenvalue of the operator $-\Delta$ on $M$.
    Then, the quadratic form in \eqref{eq:quadratic_form_Q_for_Euler} is negative definite if $-\lambda_{\text{min}}<f'<0$ everywhere, while it is positive definite if $f'>0$ everywhere.
\end{proposition}
\begin{proof}
    Assume that $f$ is invertible.
    We observe that the functional relation between $\omega_0$ and $\psi_0$ implies
    \begin{equation*}
        \{ \xi, \psi_0\} = \{ \xi, f^{-1}(\omega_0)\} =  (f^{-1})'(\omega_0) \{ \xi,\omega_0\}.
    \end{equation*}
    Thus,
    \begin{equation}
        Q(\xi) = \langle [-\Delta^{-1}+(f^{-1})'(\omega_0)]\{\xi,\omega_0 \}, \{\xi,\omega_0 \} \rangle_{L^2}.
        \label{eq:Q_xi}
    \end{equation}
    Since the $L^2$ operator norm of $-\Delta^{-1}$ is $\lambda_{\text{min}}^{-1}$, the thesis follows recalling that $(f^{-1})'=\frac{1}{f'\circ f^{-1}}$.
\end{proof}
Proposition \ref{prop:stability_geq-2} tells that $\omega_0$ is an extremum for the Hamiltonian restricted to the leaf taking into account the conservation of all the Casimirs.
If the manifold $M$ enjoys any symmetries, these might further constraint the dynamics, providing additional conserved quantities.
A relevant case is $M=\esse^2$, for which $\lambda_{\text{min}}=2$.
Here, the rotational symmetry of the sphere allows to prove the conservation of angular momentum, which is also a feature of the Zeitlin model \eqref{eq:angmom}.
More precisely, if $x_\alpha$, $\alpha\in\{1,2,3\}$ denotes the $\alpha$-th coordinate of a point on $\esse^2\subset\R^3$, the following holds:
\begin{equation} 
    \diff{}{t}\langle\omega,x_\alpha\rangle=0.
    \label{eq:cons_ang_mom_Euler}
\end{equation}
This additional constraint improves the stability threshold $f'>-2$ in Proposition~\ref{prop:stability_geq-2} to $f'>-6$ \cite[Theorem 5]{Costantin2022}.
It can be also seen from the point-of-view of the quadratic form \eqref{eq:quadratic_form_Q_for_Euler}:
\begin{proposition} \label{prop:stability_geq-6}
    Let $\omega_0 = \Delta\psi_0 = f(\psi_0)$ for some real analytic function $f:\erre\to\erre$ be a stationary solution to the 2-D Euler equations \eqref{eq:euler} on $M=\esse^2$.
    Then, the quadratic form in \eqref{eq:quadratic_form_Q_for_Euler} is negative definite if $-6<f'<0$ everywhere, while it is positive definite if $f'>0$ everywhere.
\end{proposition}
\begin{proof}
    Recalling \eqref{eq:Taylor_exp_Ad_star}, the variation can be expressed as
    \begin{equation*}
        \operatorname{Ad}^*_{\operatorname{exp}(\epsilon \xi)}\omega_0-\omega_0=\sum_{k=0}^\infty \frac{\epsilon^k}{k!} (\operatorname{ad}^*_\xi)^k\omega_0-\omega_0=\sum_{k=1}^\infty \frac{\epsilon^k}{k!} (\operatorname{ad}^*_\xi)^k\omega_0.
    \end{equation*}
    Thus, the conservation of angular momentum \eqref{eq:cons_ang_mom_Euler} yields
    \begin{equation*}
        \left\langle\sum_{k=1}^\infty \frac{\epsilon^k}{k!} (\operatorname{ad}^*_\xi)^k\omega_0\,,\,x_\alpha\right\rangle=\sum_{k=1}^\infty\left\langle \frac{1}{k!} (\operatorname{ad}^*_\xi)^k\omega_0\,,\,x_\alpha\right\rangle\epsilon^k=0\quad\forall\epsilon.
    \end{equation*}
    This implies
    \begin{equation*}
        \left\langle \frac{1}{k!} (\operatorname{ad}^*_\xi)^k\omega_0\,,\,x_\alpha\right\rangle=0\quad\forall k\geq1,
    \end{equation*}
    whence in particular for $k=1$ we get $\langle\operatorname{ad}^*_\xi\omega_0,x_\alpha\rangle=0$.
    Since the first eigenspace of $-\Delta^{-1}$ is spanned by $\{x_\alpha\}_{\alpha=1}^3$, in \eqref{eq:Q_xi} we may consider the operator $-\Delta^{-1}$ restricted to the eigenspaces of degree $\geq2$.
    In this setting, the $L^2$ operator norm of $-\Delta^{-1}$ is $1/6$, whence the thesis follows using that $(f^{-1})'=\frac{1}{f'\circ f^{-1}}$.
\end{proof}

\section{Lyapunov stability for the Zeitlin model} \label{sec:application_to_Zeitlin}

The goal of this section is to prove the stability Theorem \ref{thm:stability_informal}.
Indeed, we are going to prove the following.
\begin{theorem}[Stability] \label{thm:stability}
    Consider $W_0,P_0\in\mathfrak{su}(n)$ satisfying \eqref{eq:Zeitlin_steady}.
    Denoting by $\{\mathrm{i}p_j\}_{j=1}^n$ the eigenvalues of $P_0$ and by $\{\mathrm{i}w_j\}_{j=1}^n$ the eigenvalues of $W_0$, assume that
    \begin{equation}
        \min_{\substack{m\in\{0,...,n-1\}\\ j\in\{1,...,n-m\}}}\left\{\frac{w_{j+m}-w_j}{p_{j+m}-p_j}\right\}>-6.
        \label{eq:assumption_thm1}
    \end{equation}
    Then, the steady state $W_0,P_0$ is Lyapunov stable in the Frobenius norm.
    In particular, if $W_0,P_0$ satisfy equation~\eqref{eq:Zeitlin_steady_f} for a real analytic function $f:\mathbb{R}\to\mathbb{R}$, then assumption \eqref{eq:assumption_thm1} is fulfilled if $f'>-6$ everywhere.
\end{theorem}
The key idea of the proof is to check the definiteness of the quadratic form \eqref{eq:def_of_Q_corollary}, as done in Proposition \ref{prop:stability_geq-6}.
Given $n\in\enne$, $n\geq2$, we consider the Lie--Poisson dynamics \eqref{eq:LP} with the choice $G=\mathrm{SU}(n)$, $\mathfrak{g}=\mathfrak{su}(n)$, and $\mathfrak{g}^*\simeq \mathfrak{su}(n)$.
Here, the pairing is given by the Frobenius inner product \eqref{eq:Frobenius_inner_prod}.
We then have $\operatorname{ad}_{X}P = -\frac{1}{\hbar}[X,P]$ and $\operatorname{ad}^*_{X}W = \frac{1}{\hbar}[X,W]$, with $[\cdot,\cdot]$ being the matrix commutator.
Let the Hamiltonian in \eqref{eq:geodesic_motion} be defined in terms of the inertia operator $A = -\Delta_n^{-1}$, where $\Delta_n$ denotes the Hoppe--Yau Laplacian.
Under these assumptions, from \eqref{eq:LP} we retrieve the differential matrix system \eqref{eq:Zeitlin}.

In this setting, we now aim to prove our result Theorem \ref{thm:stability} regarding the Lyapunov stability of a steady state $W_0,P_0$ fulfilling equation~\eqref{eq:Zeitlin_steady}.
Since our setting is finite-dimensional, Lyapunov stability in the Frobenius norm immediately follows for extrema of the Hamiltonian, see for instance \cite[Chapter II, Theorem 3.3]{ArKh1998} or \cite[Lemma 1.1]{elgindi2023remarkstabilityenergymaximizers}.
Due to this fact, our plan is to rewrite Corollary \ref{cor:Arnold} in the setting of the Zeitlin model, and then check the definiteness of the obtained quadratic form.
Namely, we need to show that under the assumptions in Theorem \ref{thm:stability} the quadratic form on $\mathfrak{su}(n)$
\begin{equation}
      Q(X) = \langle[X,W_0],-\Delta_n^{-1}[X,W_0] + [X,P_0]\rangle
      \label{eq:def_of_Q_Zeitlin}
\end{equation}
is positive or negative definite.
To do so, the key idea is to perform the computations indexing by sub-diagonals, instead of indexing by rows and columns.
Namely, let us denote by $(X)_{m:j}$ the $j$-th element of the $m$-th diagonal of the matrix $X$.
This is the idea in \cite[Lemma 5.3]{modin2025briefintroductionmatrixhydrodynamics}, which we will use extensively in the remainder of the paper.
For convenience, we report the lemma here.
\begin{lemma}[\cite{modin2025briefintroductionmatrixhydrodynamics}] \label{lemma:klas}
    Consider any $X,W,P\in\mathfrak{u}(n)$ with $[P,W]=0$.
    Let $\Lambda\in\mathrm{U}(n)$ be such that both $\Lambda^\dagger W\Lambda$ and $\Lambda^\dagger P\Lambda$ are diagonal matrices.
    Denote by $\{\mathrm{i}p_j\}_{j=1}^n$ and $\{\mathrm{i}w_j\}_{j=1}^n$ the eigenvalues of $P$ and $W$ ordered according to the eigenbasis $\Lambda$.
    Then, the following formula holds:
    \begin{equation*}
        \langle [X,W],[X,P]\rangle=\frac{1}{n}\sum_{m=-n+1}^{n-1}\sum_{j=1}^{n-|m|}|(\Lambda^\dagger X\Lambda)_{m:j}|^2(p_{j+|m|}-p_j)(w_{j+|m|}-w_j).
    \end{equation*}
\end{lemma}
We are now ready to prove Theorem \ref{thm:stability}.
\begin{proof}[Proof of Theorem \ref{thm:stability}]
    Lemma \ref{lemma:klas} gives the existence of $\Lambda\in\mathrm{U}(n)$ such that
    \begin{equation}
        \langle[X,W_0],[X,P_0]\rangle=\frac{1}{n}\sum_{m=-n+1}^{n-1}\sum_{j=1}^{n-|m|}|(\Lambda^\dagger X\Lambda)_{m:j}|^2(p_{j+|m|}-p_j)(w_{j+|m|}-w_j),
        \label{eq:proof_31_1}
    \end{equation}
    where $\{\mathrm{i}p_j\}_{j=1}^n$ are the eigenvalues of $P_0$ and $\{\mathrm{i}w_j\}_{j=1}^n$ are the eigenvalues of $W_0$.
    Defining
    \begin{equation*}
        C:=\max_{\substack{m\in\{0,...,n-1\}\\ j\in\{1,...,n-m\}}}\left\{\frac{p_{j+m}-p_j}{w_{j+m}-w_j}\right\},
    \end{equation*}
    from \eqref{eq:proof_31_1} and another application of Lemma \ref{lemma:klas} we get the upper bound
    \begin{equation}
        \langle[X,W_0],[X,P_0]\rangle\leq C\|[X,W_0]\|^2.
        \label{eq:proof_31_2}
    \end{equation}
    Arguing similarly, we can also get the lower bound
    \begin{equation}
        \underbrace{\min_{\substack{m\in\{0,...,n-1\}\\ j\in\{1,...,n-m\}}}\left\{\frac{p_{j+m}-p_j}{w_{j+m}-w_j}\right\}}_{=:c}\|[X,W_0]\|^2\leq\langle[X,W_0],[X,P_0]\rangle.
        \label{eq:proof_31_3}
    \end{equation}
    Combining \eqref{eq:proof_31_2} and \eqref{eq:proof_31_3}, we see that the quadratic form $Q$ in \eqref{eq:def_of_Q_Zeitlin} satisfies
    \begin{equation}
        c \lVert [X,W_0] \rVert^2 \leq Q(X) + \langle\Delta_n^{-1}[X,W_0],[X,W_0]\rangle \leq C\lVert [X,W_0] \rVert^2.
        \label{eq:proof_31_4}
    \end{equation}
    Now, the conservation of angular momentum \eqref{eq:angmom} allows us to argue as in the proof of Proposition \ref{prop:stability_geq-6}.
    Since $-\Delta_n^{-1}$ is a positive definite operator, and its operator norm restricted to the eigenspaces of degree $\geq2$ is $1/6$, from \eqref{eq:proof_31_4} we readily get
    \begin{equation*}
        c \lVert [X,W_0] \rVert^2 \leq Q(X) \leq \left(\frac{1}{6} + C \right)\lVert [X,W_0] \rVert^2.
    \end{equation*}
    In particular, $Q$ is negative definite if $C<-1/6$ and positive definite if $c>0$.
    Assuming \eqref{eq:assumption_thm1}, the first part of Theorem \ref{thm:stability} follows.

    Regarding the case of $W_0,P_0$ satisfying \eqref{eq:Zeitlin_steady_f}, it is enough to observe that $W_0 = \mathrm{i}f(-\mathrm{i}P_0)$ implies $w_j=f(p_j)$.
    Fix $m\in\{0,...,n-1\}$ and $j\in\{1,...,n-m\}$.
    Then, the mean value theorem yields the existence of a number $\xi_{jm}\in [w_j, w_{j+m}]$ (without loss of generality, we can diagonalise $W_0$ in such a way that all the eigenvalues are sorted) such that
    \begin{equation*}
        \frac{p_{j+m}-p_j}{w_{j+m}-w_j}=\frac{f^{-1}(w_{j+m})-f^{-1}(w_j)}{w_{j+m}-w_j}=(f^{-1})'(\xi_{im}).
    \end{equation*}
    This concludes the proof of Theorem \ref{thm:stability}.
\end{proof}

\section{Rigidity results} \label{sec:rigidity}

The goal of this section is to prove the rigidity Theorem \ref{thm:rigidity_informal}.
To do so, a few comments are necessary to clarify the statement concerning the action of the symmetry group $\mathrm{SO}(3)$. 
Indeed, a remarkable feature of Zeitlin's model on the sphere is that it preserves the rotational symmetry corresponding to the action of the Lie group $\mathrm{SO}(3)$ on $\mathbb{S}^2$.
In the matrix case, the action of $\mathrm{SO}(3)$ on $\mathfrak{su}(n)$ (corresponding to the induced action of $\mathrm{SO}(3)$ on $\mathcal C^\infty(\mathbb{S}^2)$) is generated by the matrices $X_1,X_2,X_3$ (which indeed give a unitary representation of $\mathfrak{so}(3)$).
The action is constructed in the following way: for a rotational matrix $R\in \mathrm{SO}(3)$, first compute its logarithm $\rho = \log R \in \mathfrak{so}(3)$.
We have $\mathfrak{so}(3)\simeq \mathbb{R}^3$, and we denote by $\{\rho_\alpha\}_{\alpha=1}^3$ the components of $\rho$ with respect to the canonical basis.
The action of $R$ on $W\in \mathfrak{su}(n)$ is then
\[
R\cdot W = F_R W F_R^\dagger,
\]
where $$F_R = \exp\left(\sum_{\alpha=1}^3 \rho_\alpha X_\alpha\right).$$
For more details, see~\cite{MoVi2026}.

Consider now $W_0,P_0\in\mathfrak{su}(n)$ satisfying \eqref{eq:Zeitlin_steady}.
Denote by $\{\mathrm{i}p_j\}_{j=1}^n$ the eigenvalues of $P_0$ and by $\{\mathrm{i}w_j\}_{j=1}^n$ the eigenvalues of $W_0$, as in Lemma \ref{lemma:klas}, and assume that \eqref{eq:assumption_thm1} holds.
The rigidity results in Theorem~\ref{thm:rigidity_informal} are captured in the following two propositions.
\begin{proposition}\label{prop:rigidity_-6}
    Assume that
    \begin{equation}
        \min_{\substack{m\in\{0,...,n-1\}\\ j\in\{1,...,n-m\}}}\left\{\frac{w_{j+m}-w_j}{p_{j+m}-p_j}\right\}>-6.
        \label{eq:assumption_thm1_prop}
    \end{equation}
    Then there exists $R\in\mathrm{SO}(3)$ such that $R\cdot W_0$ is a diagonal matrix.
\end{proposition}
\begin{proposition}\label{prop:rigidity_-2}
    Assume that
    \begin{equation}
        \min_{\substack{m\in\{0,...,n-1\}\\ j\in\{1,...,n-m\}}}\left\{\frac{w_{j+m}-w_j}{p_{j+m}-p_j}\right\}>-2.
        \label{eq:min>-2}
    \end{equation}
    Then, $W_0=0$.
\end{proposition}

The remainder of the section is devoted to the proofs of Propositions \ref{prop:rigidity_-6} and \ref{prop:rigidity_-2}, in Section \ref{subsec:f_minus_6} and \ref{subsec:f_minus_2} respectively.
The argument that we employ is inspired to \cite[Theorem 4]{Costantin2022}.

\subsection{Proof of Proposition \ref{prop:rigidity_-6}} \label{subsec:f_minus_6}

Consider the steady state $W_0,P_0$ \eqref{eq:Zeitlin_steady}.
We begin by observing that it is enough to show $[R\cdot W_0,X_3]=0$ for some $R\in\mathrm{SO}(3)$.
In this direction, we may exploit the following property of the first nontrivial eigenspace of the Hoppe--Yau Laplacian $\Delta_n$, that coincides with $\operatorname{Span}\{X_\alpha\}_{\alpha=1}^3$.
Namely, there exists $R\in\mathrm{SO}(3)$ such that $\mathbb{P}_1(R\cdot W_0)$ and $\mathbb{P}_1(R\cdot P_0)$ are diagonal matrices.
One can also check that the operator $R\cdot$ commutes with the operator $\Delta_n$, and that the matrix $R\cdot W_0$ commutes with the matrix $R\cdot P_0$.
Furthermore, since $\Delta_n$ commutes with the operator $[X_3,\cdot]$, we can write
\begin{equation}
    \langle[R\cdot W_0,X_3],[R\cdot P_0,X_3]\rangle=\langle[\Delta_nR\cdot P_0,X_3],[R\cdot P_0,X_3]\rangle=\langle\Delta_n[R\cdot P_0,X_3],[R\cdot P_0,X_3]\rangle.
    \label{eq:proof_41_A}
\end{equation}
Since $[R\cdot P_0,R\cdot W_0]=0$, we can apply Lemma \ref{lemma:klas} to infer the existence of a matrix $\Lambda\in\mathrm{U}(N)$ such that
\begin{equation}
    \langle[R\cdot W_0,X_3],[R\cdot P_0,X_3]\rangle=\frac{1}{n}\sum_{m=-n+1}^{n-1}\sum_{j=1}^{n-|m|}|(\Lambda^\dagger X_3\Lambda)_{m:j}|^2(p_{j+|m|}-p_j)(w_{j+|m|}-w_j),
    \label{eq:proof_41_1}
\end{equation}
where $\{\mathrm{i}p_j\}_{j=1}^n$ are the eigenvalues of $P_0$ and $\{\mathrm{i}w_j\}_{j=1}^n$ are the eigenvalues of $W_0$, coherently with the notation of Lemma \ref{lemma:klas} (recall that rotations preserve the eigenvalues).
Now, assume by contradiction that
\begin{align}
\begin{split}
    \exists(m,j)\in\{-n+1,\dots,n-1\}\times\{1,\dots,N\}\\
    \text{with }j\leq N-|m|\quad\text{such that}\quad|(\Lambda^\dagger X_3\Lambda)_{m:j}|^2(p_{j+|m|}-p_j)^2>0.
    \label{eq:proof_41_2}
\end{split}
\end{align}
Together with \eqref{eq:assumption_thm1_prop}, assumption \eqref{eq:proof_41_2} implies
\begin{align}
\begin{split}
    \sum_{m=-n+1}^{n-1}\sum_{j=1}^{n-|m|}|(\Lambda^\dagger X_3&\Lambda)_{m:j}|^2(p_{j+|m|}-p_j)(w_{j+|m|}-w_j)\\
    &>-6\sum_{m=-n+1}^{n-1}\sum_{j=1}^{n-|m|}|(\Lambda^\dagger X_3\Lambda)_{m:j}|^2(p_{j+|m|}-p_j)^2.
    \label{eq:proof_41_3}
\end{split}
\end{align}
Applying again Lemma \ref{lemma:klas}, we observe that the right-hand side of \eqref{eq:proof_41_3} is equal to $-6\|[R\cdot P_0,X_3]\|^2$.
Taking this fact into account, combining \eqref{eq:proof_41_1} and \eqref{eq:proof_41_3} we infer
\begin{equation}
    \langle[R\cdot W_0,X_3],[R\cdot P_0,X_3]\rangle>-6\|[R\cdot P_0,X_3]\|^2.
    \label{eq:proof_41_B}
\end{equation}
Next, the anti-cyclic property of the trilinear form $\langle[\cdot,\cdot],\cdot\rangle$ and the fact that $[X_i,X_j]=\hbar\epsilon_{ijk}X_k$ yield
\begin{equation}
    \mathbb{P}_1[R\cdot P_0,X_3]=\frac{3}{n}\sum_{\alpha=1}^3\langle[R\cdot P_0,X_3],X_\alpha\rangle X_\alpha=\frac{3\hbar}{n}\biggl(\langle X_2,R\cdot P_0\rangle X_1-\langle X_1,R\cdot P_0\rangle X_2\biggr).
    \label{eq:proof_41_4}
\end{equation}
Since $[\mathbb{P}_1R\cdot P_0,X_3]=0$, we have that $\mathbb{P}_1(R\cdot P_0)=a_3X_3$ for some $a_3\in\R$.
This information, together with the fact that $X_1$ and $X_2$ belong to the first eigenspace of $\Delta_n$, gives
\begin{align}
\begin{split}
    \langle X_2,R\cdot P_0\rangle X_1-\langle X_1,R\cdot P_0\rangle X_2=&\langle X_2,\mathbb{P}_1(R\cdot P_0)\rangle X_1-\langle X_1,\mathbb{P}_1(R\cdot P_0)\rangle X_2\\
    =&a_3\underbrace{\langle X_2,X_3\rangle X_1}_{=0}-a_3\underbrace{\langle X_1,X_3\rangle}_{=0}X_2=0.
    \label{eq:proof_41_5}
\end{split}
\end{align}
Combining \eqref{eq:proof_41_4} and \eqref{eq:proof_41_5} we get $\mathbb{P}_1[R\cdot P_0,X_3]=0$, whence
\begin{equation}
    \langle\Delta_n[R\cdot P_0,X_3],[R\cdot P_0,X_3]\rangle\leq-6\|[R\cdot P_0,X_3]\|^2.
    \label{eq:proof_41_C}
\end{equation}
Finally, combining \eqref{eq:proof_41_A}, \eqref{eq:proof_41_B} and \eqref{eq:proof_41_C} we get a contradiction, thus \eqref{eq:proof_41_2} does not hold true.
That is, in view of Lemma \ref{lemma:klas}, $[R\cdot P_0,X_3]=0$.
Since the operator $\Delta_n$ commutes with both $R\cdot$ and $[X_3,\cdot]$, we also have $[R\cdot W_0,X_3]=0$.
This concludes the Proof of Proposition \ref{prop:rigidity_-6}.

\subsection{Proof of Proposition \ref{prop:rigidity_-2}} \label{subsec:f_minus_2}

Let $W_0,P_0$ be a steady state \eqref{eq:Zeitlin_steady}.
Using the characterisation of the Hoppe--Yau Laplacian in terms of the generators $X_\alpha$'s, and the anti-cyclic property of the trilinear form $\langle[\cdot,\cdot],\cdot\rangle$. we can write
\begin{equation}
    \|\Delta_nP_0\|^2=\langle\Delta_nP_0,\Delta_nP_0\rangle=\frac{1}{\hbar^2}\sum_{\alpha=1}^3\langle[X_\alpha,[X_\alpha,P_0]],\Delta_nP_0\rangle=-\frac{1}{\hbar^2}\sum_{\alpha=1}^3\langle[X_\alpha,W_0],[X_\alpha,P_0]\rangle.
    \label{eq:proof_42_1}
\end{equation}
Next, notice that we can appeal to Lemma \ref{lemma:klas} to compute the right-hand side of \eqref{eq:proof_42_1}:
\begin{equation}
    \sum_{\alpha=1}^3\langle[X_\alpha,W_0],[X_\alpha,P_0]\rangle=\frac{1}{n}\sum_{\alpha=1}^3\sum_{m=-n+1}^{n-1}\sum_{j=1}^{n-|m|}|(\Lambda^\dagger X_\alpha \Lambda)_{m:j}|^2(p_{j+|m|}-p_j)(w_{j+|m|}-w_j),
    \label{eq:proof_42_2}
\end{equation}
where $\{\mathrm{i}p_j\}_{j=1}^n$ are the eigenvalues of $P_0$ and $\{\mathrm{i}w_j\}_{j=1}^n$ are the eigenvalues of $W_0$, coherently with the notation of Lemma \ref{lemma:klas}.
By contradiction, assume that
\begin{align}
\begin{split}
    \exists(\alpha,m,j)\in\{1,2,3\}\times\{-n+1,\dots,n-1\}\times\{1,\dots,n\}\\
    \text{with }j\leq n-|m|\quad\text{such that}\quad|(\Lambda^\dagger X_\alpha\Lambda)_{m:j}|^2(p_{j+|m|}-p_j)^2>0.
    \label{eq:to_reach_contrd}
\end{split}
\end{align}
Together with \eqref{eq:min>-2}, assumption \eqref{eq:to_reach_contrd} implies
\begin{align}
\begin{split}
    \sum_{\alpha=1}^3\sum_{m=-n+1}^{n-1}\sum_{j=1}^{n-|m|}|(\Lambda^\dagger &X_\alpha \Lambda)_{m:j}|^2(p_{j+|m|}-p_j)(w_{j+|m|}-w_j)\\
    &>-2\sum_{\alpha=1}^3\sum_{m=-n+1}^{n-1}\sum_{j=1}^{n-|m|}|(\Lambda^\dagger X_\alpha \Lambda)_{m:j}|^2(p_{j+|m|}-p_j)^2.
\label{eq:proof_42_3}
\end{split}
\end{align}
Applying again Lemma \ref{lemma:klas}, we observe that the right-hand side of \eqref{eq:proof_42_3} is equal to $-2\sum_{\alpha=1}^3\|[X_\alpha,P_0]\|^2$.
Taking this fact into account, combining \eqref{eq:proof_42_1}, \eqref{eq:proof_42_2} and \eqref{eq:proof_42_3} we infer
\begin{equation}
    \|\Delta_nP_0\|^2<\frac{2}{\hbar^2}\sum_{\alpha=1}^3\|[X_\alpha,P_0]\|^2.
    \label{eq:proof_42_4}
\end{equation}
Finally, to compute the right-hand side of \eqref{eq:proof_42_4} we appeal again to the properties of the trilinear form $\langle[\cdot,\cdot],\cdot\rangle$ and the expression for $\Delta_n$ in terms of the $X_\alpha$'s.
This gives
\begin{align}
\begin{split}
    \frac{2}{\hbar^2}\sum_{\alpha=1}^3\|[X_\alpha,P_0]\|^2=&\frac{2}{\hbar^2}\sum_{\alpha=1}^3\langle[X_\alpha,P_0],[X_\alpha,P_0]\rangle\\
    =&-\frac{2}{\hbar^2}\sum_{\alpha=1}^3\langle[X_\alpha,[X_\alpha,P_0]],P_0\rangle=-2\langle\Delta_nP_0,P_0\rangle.
    \label{eq:proof_42_5}
\end{split}
\end{align}
Expanding in terms of the orthonormal basis $\{T_{l,m}^n\}$ of $\mathfrak{su}(n)$ \cite{Hoppe1982}, we observe that
\begin{align}
\begin{split}
    -2\langle\Delta_nP_0,P_0\rangle=&\sum_{l=1}^{n-1}2l(l+1)\sum_{m=-l}^l|\langle P_0,T_{l,m}^n\rangle|^2\\
    \leq&\sum_{l=1}^{n-1}l^2(l+1)^2\sum_{m=-l}^l|\langle P_0,T_{l,m}^n\rangle|^2=\|\Delta_nP_0\|^2.
    \label{eq:proof_42_6}
\end{split}
\end{align}
Combining \eqref{eq:proof_42_4}, \eqref{eq:proof_42_5} and \eqref{eq:proof_42_6}, we reach a contradiction.
Hence, \eqref{eq:to_reach_contrd} does not hold true, i.e., $[X_\alpha,P_0]=0$ appealing again to Lemma \ref{lemma:klas}.
Since this holds for all $\alpha=1,2,3$, we infer $W_0=\Delta_nP_0=0$.
This concludes the proof of Proposition \ref{prop:rigidity_-2}.

\section*{Acknowledgments}
We would like to thank Michele Coti Zelati, Boris Khesin, Michael Roop, and Erik Wahlén for guidance and stimulating discussions in the early stages of this project.
The research of LM was supported by EPSRC under grant number EP/Y03533X/1, with additional support from the Department of Mathematics at Imperial College London.
The research of KM was supported by the Swedish Research Council (grant number 2022-03453), the Knut and Alice Wallenberg Foundation (grant number KAW2024.0229), and the Göran Gustafsson Foundation for Research in Natural Sciences and Medicine.

\bibliographystyle{abbrvnat}
\bibliography{biblio}
\end{document}